\documentclass[10pt,reqno,sumlimits]{amsart} 
\usepackage{amssymb,amscd, epsfig, mathrsfs, xypic, amsmath,amsthm, color, enumerate} 
\xyoption{all}  
\usepackage[dvipdfmx]{hyperref}
\usepackage{eepic, xcolor}
\definecolor{grey}{rgb}{0.86, 0.86, 0.86} 
\definecolor{red2}{rgb}{0.86, 0.66, 0.86} 
\newcommand{\bfb}{{\mathbf b}}

\newcommand{\bfk}{{\mathbf k }}

\newcommand{\bfp}{{\mathbf p}}
\newcommand{\bfq}{{\mathbf q }}

\newcommand{\bfu}{{\mathbf u}}
\newcommand{\bfv}{{\mathbf v}}

\newcommand{\NN}{{\mathbb N}}

\newcommand{\ZZ}{{\mathbb Z}}

\newcommand{\frakC}{{\mathfrak C}}

\newcommand{\calP}{{\mathcal P}}

\newcommand{\calS}{{\mathcal S}}

\newcommand{\bfP}{{\mathbf P}}

\newcommand{\scP}{{\mathscr P}}

\newcommand{\sign}{{\operatorname{sign}}}

\newcommand{\sgn}{{\operatorname{sgn}}}

\newcommand{\SP}{\calS\!\calP}

\newcommand{\SST}{\operatorname{SST}}
\newcommand{\Pf}{\operatorname{Pf}}
\newcommand{\MST}{\operatorname{MST}}

\newcommand{\GF}{{GF}}

\newtheorem{thm}{Theorem}[section]  
 
\newtheorem{lem}[thm]{Lemma}  
\newtheorem{prop}[thm]{Proposition} 
\newtheorem{df-pr}[thm]{Definition-Proposition}

\theoremstyle{definition} 
\newtheorem{defn}[thm]{Definition}
   
\newtheorem{rem}[thm]{Remark}
 
\newtheorem{exm}[thm]{Example}

\numberwithin{equation}{section} 
\begin{document} 
\title{A tableau formula for vexillary Schubert polynomials in type C}
\author{Tomoo Matsumura}
\maketitle 
\begin{abstract}
Ikeda--Mihalcea--Naruse's double Schubert polynomials \cite{IkedaMihalceaNaruse} represent the equivariant cohomology classes of Schubert varieties in the type C flag varieties. The goal of this paper is to obtain a new tableau formula of these polynomials associated to {\it vexillary signed permutations} introduced by Anderson--Fulton \cite{AndersonFulton2}. To achieve that goal, we introduce {\it flagged factorial (Schur) $Q$-functions}, combinatorially defined functions in terms of {\it marked shifted tableaux for flagged strict partitions}, and prove their Schur--Pfaffian formula.  As an application, we also obtain a new combinatorial formula of factorial $Q$-functions of Ivanov \cite{Ivanov04} in which {\it monomials} bijectively correspond to flagged marked shifted tableaux. 
\end{abstract}


\section{Introduction}
Ikeda--Mihalcea--Naruse \cite{IkedaMihalceaNaruse} introduced the {\it double Schubert polynomials} of type $C$ (also $B$ and $D$) by extending Billey--Haiman \cite{BilleyHaiman}'s construction for the single case. These polynomials  represent the equivariant cohomology classes of Schubert varieties in type $C$ flag varieties. One can express the ones corresponding to Lagrangian Grassmannians in terms of the Schur-Pfaffian by the work of Kazarian \cite{Kazarian} and Ikeda \cite{Ikeda2007}, and they coincide with the {\it factorial Schur $Q$-functions} of Ivanov \cite{Ivanov04} defined in terms of {\it marked shifted tableaux} of strict partitions (cf. \cite{MacdonaldHall}). Note that the corresponding fact for the single case was established in the earlier work of Pragacz \cite{PragaczPQ}. In \cite{AndersonFulton, AndersonFulton2}, Anderson--Fulton introduced {\it vexillary signed permutations}, a family of signed permutations containing the Lagrangian ones, and showed that the associated double Schubert polynomials can be also expressed in the Schur-Pfaffian formula. The goal of this paper is to give a new tableau formula for this family of double Schubert polynomials, by extending the notion of marked shifted tableaux and Ivanov's factorial Schur $Q$-functions.

Our study is motivated by the analogy in type A. Lascoux--Sch\"{u}tzenberger's double Schubert polynomials \cite{ClassesLascoux, SchubertLascoux, LascouxSchutzenberger1985} represent the equivariant cohomology classes of Schubert varieties of type $A$ flag varieties, due to Fulton \cite{FlagsFulton}. A family of permutations including Grassmannian ones, now called {\it vexillary permutations}, was singled out by Lascoux and their associated double Schubert polynomials are given in a Jacobi--Trudi type determinant formula.  It is worth mentioning that the ones associated to Grassmannians coincide with the factorial Schur polynomials essentially introduced and studied by Biedenharn--Louck \cite{BiedenharnLouck}. The flagged double (or factorial) Schur polynomials generalize such kind of double Schubert polynomials and are defined either by flagged determinant formula or by flagged semistandard tableaux for a partition by the work of Chen--Li--Louck \cite{ChenLiLouck} (for the single case, see Gessel--Viennot \cite{GesselViennot}, and Wachs \cite{Wachs}).

Below we explain our main results in more detail. Let $\lambda=(\lambda_1,\dots,\lambda_r)$ be a strict partition of length $r$, {\it i.e.}, a strictly decreasing sequence of  $r$ positive integers. We identify it with its {\it shifted Young diagram}, obtained from the usual Young diagram by shifting the $i$-th row $(i-1)$ boxes to the right, for each $i\geq 1$. Let $f=(f_1,\dots, f_r)$ be a sequence of nonnegative integers. We call $f$ a  {\it flagging} of $\lambda$ and the pair $(\lambda,f)$ a {\it flagged strict partition}. Consider the ordered set of alphabets: {\it unmarked} numbers $1,2,\dots$ and {\it primed} numbers $1',2',\dots$ with $1'<1<2'<2<\cdots$. The classical {\it marked shifted tableau} $T$ of $\lambda$ is an assignment of such an alphabet to each box of the diagram subject to the rules: (1) assigned alphabets are weakly increasing in each column and row; (2) unmarked numbers are strictly increasing in each column; (3) primed numbers are strictly increasing in each row. In order to extend this notion, we add, to the above ordered set of alphabets, {\it circled} numbers $1^{\circ}<2^{\circ}<\cdots $ which are greater than any unmarked and primed number. We define a {\it (flagged) marked shifted tableau} of
$(\lambda,f)$ to be an assignment of an alphabet to each box of $\lambda$ with rules: in addition to (1), (2), and (3), we require (4) circled numbers are strictly increasing in each row, and (5) alphabets in the $i$-th row are at most $f_i^{\circ}$. We denote the set of all marked shifted tableaux of $(\lambda,f)$ by $\MST(\lambda,f)$.

A {\it signed permutation} $w$ is a permutation on the set $\{1,2,\dots \} \cup \{-1,-2,\dots\}$ such that $w(i)\not= i$ for only finitely many $i$, and $\overline{w(i)}=w(\bar i)$ where we denote $\bar i = -i$. Let $x=(x_i)_{i\in \NN}, z=(z_i)_{i\in \NN}, b=(b_i)_{i\in \NN}$. The double Schubert polynomial associated to a signed permutation $w$ is denoted by $\frakC_w(x;z|b)$. Note that the variables $b$ coincide with $-t$ in the notation of \cite{IkedaMihalceaNaruse}. 

If a signed permutation $w$ is {\it vexillary} in the sense of Anderson--Fulton \cite{AndersonFultonVex}, there is a unique flagged strict partition $(\lambda,f)$. For each $T \in \MST(\lambda,f)$, we assign
\[
(xz|b)^T = \prod_{k \in T}\left(x_k+b_{c(k)-r(k)}\right)\cdot \prod_{k' \in T}\left(x_k-b_{c(k')-r(k')}\right) \cdot \prod_{k^{\circ} \in T} \left(z_k+b_{k+r(k^{\circ})-c(k^{\circ})}\right),
\]
where $r(\ )$ and $c(\ )$ denote the column and row indices of the entry respectively, and we set $b_{-i}:=-b_{i+1}$ for all $i\geq0$. 

Our main result is as follows.

\vspace{3mm}

\noindent{\bf Theorem A} (Theorem \ref{thmmain2}).
{\it
Let $w$ be a vexillary signed permutation in the sense of Anderson--Fulton \cite{AndersonFultonVex} and $(\lambda,f)$ the corresponding flagged strict partition. Then we have
\begin{equation}\label{eqintro1}
\frakC_w(x;z|b) = \sum_{T\in \MST(\lambda,f)} (xz|b)^T.
\end{equation}
}

\vspace{0mm}

For a general flagged strict partition $(\lambda,f)$, we denote by $Q_{\lambda,f}(x;z|b)$ the function defined by the right hand side of (\ref{eqintro1}). We call it a {\it flagged factorial $Q$-function}, since it is nothing but the original definition of Ivanov's factorial $Q$-functions $Q_{\lambda}(x|b)$ when $f=(0,\dots,0)$. The  proof of Theorem A is based on the following Schur--Pfaffian formula of $Q_{\lambda,f}(x;z|b)$ which generalizes the corresponding formula of $Q_{\lambda}(x|b)$ when $f=(0,\dots,0)$ in \cite[Theorem 9.1]{Ivanov04}.

\vspace{3mm}

\noindent{\bf Theorem B} (Theorem \ref{mainthm}).
{\it
Let $(\lambda,f)$ be a flagged strict partition of length $r$. Suppose that $0<\lambda_i-f_i\leq \lambda_j-f_j$ for all $i<j$. The we have
\[
Q_{\lambda,f}(x,z|b)=\Pf\left[q_{\lambda_1}^{[f_1|\lambda_1-f_1-1]}q_{\lambda_2}^{[f_2|\lambda_2-f_2-1]}\dots q_{\lambda_r}^{[f_r|\lambda_r-f_r-1]}\right],
\]
where $\Pf$ is the Schur--Pfaffian defined in \S \ref{secPf}, and the function $q_m^{[k|\ell]}=q_m^{[k|\ell]}(x;z|b)$ is  defined by
\[
\sum_{m\geq 0} q_m^{[k|\ell]} u^m := \left(\prod_{i\geq 1} \frac{1+x_iu}{1-x_iu}\right)e^{[k]}_u(z) e^{[\ell]}_u(b), \ \ \ \ 
e^{[k]}_u(z)=\begin{cases}
\displaystyle\prod_{i=1}^k(1+z_i u) & (k\geq 0),\\
\displaystyle\prod_{i=1}^{|k|}\frac{1}{1-z_i u}& (k\leq 0).
\end{cases}
\]
}

\vspace{0mm}

As an application of Theorem A, we obtain a new tableau formula of Ivanov's factorial $Q$-function $Q_{\lambda}(x|b)$. Ikeda--Mihalcea--Naruse \cite{IkedaMihalceaNaruse} showed that $\frakC_w(x;z|b) = \frakC_{w^{-1}}(x;b|z)$ and Anderson--Fulton \cite{AndersonFultonVex} showed that if $w$ is vexillary, then so is $w^{-1}$. If $w$ is Lagrangian with strict partition $\lambda$ of length $r$, we can see that the strict partition of $w^{-1}$ is also $\lambda$ and its flag is $f=(\lambda_1-1, \dots, \lambda_r-1)$. All together we obtain

\vspace{3mm}

\noindent{\bf Theorem C} (Theorem \ref{thmmain3}).
{\it
Let $\lambda$ be a strict partition of length $r$ and $f=(\lambda_1-1,\dots,\lambda_r-1)$, then we have
\[
Q_{\lambda}(x|b) =  \sum_{T\in \MST(\lambda,f)} (xb)^T
\]
where $(xb)^T$ is the monomial given by
\[
(xb)^T=\prod_{k\in T} x_k \prod_{k'\in T} x_k  \prod_{k^{\circ}\in T} b_k.
\]
}
Anderson--Fulton \cite{AndersonFulton2} also introduced a larger family of {\it theta-vexillary signed permutations} (Lambert \cite{Lambert}), containing the $k$-Grassmannian signed permutations. They obtained the theta-polynomial (or raising operator, Pfaffian-sum) formula of double Schubert polynomials associated to such elements, extending the ones for $k$-Grassmannians (\cite{BuchKreschTamvakis1, WilsonThesis, IkedaMatsumura}). The combinatorial aspect of these signed permutation is far more complicated than the vexillary ones. In particular, it is worth mentioning that there is a tableau formula of the corresponding {\it single} Schubert polynomials associated to $k$-Grassmannian signed permutations, due to Tamvakis \cite{Tamvakis2011Crelle}. 
Since some of those polynomials can also be given in terms of the tableaux introduced in this paper, it is an interesting problem to find the relation to these expressions and to extend the formula to all theta-vexillary double Schubert polynomials. 

This paper is organized as follows. 
In Section \ref{secprelim}, we introduce a few basic functions and set up an algebraic framework to study double Schubert polynomials and the combinatorially defined functions defined in this paper. 
In Section \ref{secffQ}, we introduce flagged marked shifted tableaux and the functions defined by them. We prove a few basic formulas that will be used in the proof of Theorem B.
In Section \ref{secPf}, we review the definition of Schur-Pfaffian and prove Theorem B.
In Section \ref{secSchPol}, we first recall the basic fact about double Schubert polynomials and vexillary signed permutations, following Ikeda--Mihalcea--Naruse \cite{IkedaMihalceaNaruse} and Anderson--Fulton \cite{AndersonFultonVex}. We explain how Theorem A and Theorem C follow from Theorem B. 
In the appendix, we give a proof of a Jacobi--Trudi type formula of row-strict skew Schur polynomials, extending the work of Wachs \cite{Wachs} and Chen--Li--Louck \cite{ChenLiLouck}. This formula is used in the proof of Theorem B. 

\section{Preliminary}\label{secprelim}
Before we proceed with our main object of interest, we prepare the notations for a few basic functions. The goal is to set an algebraic framework in which we can study combinatorially defined functions. In particular, our Pfaffian formula of the vexillary double Schubert polynomials (and also the factorial flagged $Q$-functions) will be in terms of the basic functions that we review here.  We use infinite sequences of variables,  $x=(x_i)_{i\in \NN}, z=(z_i)_{i\in \NN}$, and $b=(b_i)_{i\in \NN}$.

We define functions $q_m=q_m(x)$ in the $x$-variables for integers $m\geq 0$ by the generating function 
\[
q_u(x)= \sum_{m\geq 0} q_m(x) u^m := \prod_{i\geq 1} \frac{1+x_iu}{1-x_iu},
\]
where $u$ is a formal variable. For each integer $k$, we also define polynomials $e^{[k]}_m(b)$ in the $b$-variables for $m\geq 0$ by
\[
e^{[k]}_u(b)=\sum_{m\geq 0} e^{[k]}_m(b)u^m:=\begin{cases}
\displaystyle\prod_{i=1}^k(1+b_i u) & (k\geq 0)\\
\displaystyle\prod_{i=1}^{|k|}\frac{1}{1-b_i u}& (k\leq 0).
\end{cases}
\]
The polynomials $e^{[k]}_m(b)$ and $e^{[-k]}_m(b)$ are nothing but the elementary and complete symmetric polynomials of degree $m$ in $b_1,\dots, b_k$ respectively. 

For integers $k, \ell\in \ZZ$, we set
\begin{eqnarray*}
e^{[k|\ell]}_u(z|b)&=&\sum_{m\geq 0} e_m^{[k|\ell]}(z|b) u^m := e^{[k]}_u(z) e^{[\ell]}_u(b),\\
q^{[\ell]}_u(x|b)&=&\sum_{m\geq 0} q_m^{[\ell]}(x|b) u^m := q_u(x)e^{[\ell]}_u(b),\\
q^{[k|\ell]}_u(x;z|b)&=&\sum_{m\geq 0} q_m^{[k|\ell]}(x;z|b) u^m := q_u(x)e^{[k]}_u(z) e^{[\ell]}_u(b).
\end{eqnarray*}
We will also denote $h^{[k|\ell]}_m(z|b):=e^{[-k|-\ell]}_m(z|b)$. Moreover, we often suppress the variables when it is clear from the context, {\it e.g.}, $e^{[-k|-\ell]}_m=e^{[-k|-\ell]}_m(z|b)$, $q_m^{[k|\ell]}=q_m^{[k|\ell]}(x;z|b)$, and so on. 

Occasionally we use the infinite sequence of variables $\bfb=(b_i)_{i\in \ZZ}$. With this extended sequence of $b$-variables in mind, we will use the following index shifting operator $\tau$. For each integer $k\in \ZZ$, let $\tau^k(b)$ be the sequence of variables defined by
\[
\tau^k(b) = (b_{1+k},b_{2+k},b_{3+k},\dots).
\]
Similarly $\tau^k(\bfb)$ denotes the sequence of variables such that its $i$-th variable is $b_{i+k}$ for $i\in \ZZ$.

We consider the ring $\Gamma=\ZZ[q_1,q_2,\dots]$. We should note that this is not a polynomial ring since $q_i$'s are not algebraically independent. It is well-known that $\Gamma$ has a $\ZZ$-basis consisting of Schur $Q$-functions $Q_{\lambda}(x)$ (cf. \cite{MacdonaldHall}). It is also worth mentioning that Ivanov's factorial $Q$-functions $Q_{\lambda}(x|b)$ \cite{Ivanov04} form a $\ZZ[b]$-basis of the $\ZZ[b]$-algebra $\Gamma[b]:=\Gamma\otimes_{\ZZ}\ZZ[b]$ where $\ZZ[b]$ denotes the polynomial ring in $b$-variables. All functions defined above are regarded as elements of 
\[
\Gamma[z,b]:= \Gamma \otimes_{\ZZ} \ZZ[z]\otimes_{\ZZ} \ZZ[b]. 
\]

\section{Flagged factorial $Q$-functions}\label{secffQ}
In this section, we introduce {\it flagged factorial $Q$-functions} $Q_{\lambda,f}(x;z|b)$ based on the notion of {\it marked shifted tableaux} of flagged strict partitions $(\lambda,f)$. We will also discuss basic formulas that will be used in the proof of Schur--Pfaffian formula for $Q_{\lambda,f}(x;z|b)$ in the next section.

\subsection{Definition of tableaux and functions}\label{fst}
A {\it strict partition} $\lambda=(\lambda_1,\lambda_2,\dots)$ is a sequence of non-negative integers such that $\lambda_i>\lambda_{i+1}$ if $\lambda_i\not=0$ and the number of positive integers in $\lambda$, called the {\it length} of $\lambda$, is finite. We also denote a strict partition of length $r$ as a finite sequence of $r$ positive integers $\lambda=(\lambda_1,\dots,\lambda_r)$ and identify it with its {\it shifted Young digram}, obtained from the usual Young diagram by shifting the $i$-th row $(i-1)$ boxes to the right, for $1 \leq i\leq r$. Let $\SP$ be the set of all strict partitions and $\SP_r$ the set of all strict partition of length at most $r$.

Consider the order set $\bfP$ of {\it alphabets}, consisting of {\it unmarked} numbers $1,2,\dots$, {\it primed} numbers $1',2',\dots$, and {\it circled} numbers $1^{\circ},2^{\circ},\dots$, where the total order is given by
 \[
 1'<1<2'<2 < 3'<3<\cdots < 1^{\circ}<2^{\circ}<\cdots.
 \]
For a given strict partition $\lambda$ of length $r$, a {\it flagging} of $\lambda$ is a sequence $f=(f_1,\dots, f_r)$ of non-negative integers. We call the pair $(\lambda, f)$ a {\it flagged strict partition}.  

\begin{defn}
A {\it (flagged) marked shifted tableau} of a flagged strict partition $(\lambda,f)$ is a filling of the shifted Young diagram of $\lambda$ which assigns an alphabet in $\bfP$ to each box, subject to the rules
\begin{enumerate}
\item alphabets are weakly increasing in each row and column,
\item unmarked numbers are strictly increasing in each column,
\item primed numbers are strictly increasing in each row,
\item circled numbers are strictly increasing in each row, and, 
\item for $1\leq i \leq r$,  one can assign alphabets at most $f_i^{\circ}$ in the $i$-th row.
\end{enumerate}
\end{defn}
\begin{rem}\label{remMST}
It is worth noting that, by the total order of $\bfP$ and the rule (1),  the part consisting of unmarked and primed numbers forms the usual marked shifted tableaux of the shifted Young diagram of a strict partition (cf. \cite[p.256]{MacdonaldHall}). It is also clear from the order of $\bfP$ that the part consisting of circled numbers forms a row-strict semistandard Young tableau of a skew shape $\lambda/\mu$ given by  a strict partition $\mu \subset \lambda$.
\end{rem}
\begin{exm}
Let $\lambda=(5,3,1)$ and $f=(2,1,0)$. The following are examples of marked shifted tableaux of the flagged strict partition $(\lambda,f)$:
\setlength{\unitlength}{0.5mm}
\begin{center}
\begin{picture}(60,40)
\put(00,30){\line(1,0){50}}\put(00,20){\line(1,0){50}}\put(10,10){\line(1,0){30}}\put(20,00){\line(1,0){10}}\put(00,30){\line(0,-1){10}}\put(10,30){\line(0,-1){20}}\put(20,30){\line(0,-1){30}}\put(30,30){\line(0,-1){30}}\put(40,30){\line(0,-1){20}}\put(50,30){\line(0,-1){10}}
\put(03,22){{\small $1$}}\put(13,22){{\small $2'$}}\put(23,22){{\small $2$}}\put(33,22){{\small $2$}}\put(43,22){{\small $3'$}}
                                       \put(13,12){{\small $2'$}}\put(23,12){{\small $3$}}\put(33,12){{\small $4$}}
                                                                              \put(23,02){{\small $4'$}}
\end{picture}
\ \ 
\begin{picture}(60,40)
\put(00,30){\line(1,0){50}}\put(00,20){\line(1,0){50}}\put(10,10){\line(1,0){30}}\put(20,00){\line(1,0){10}}\put(00,30){\line(0,-1){10}}\put(10,30){\line(0,-1){20}}\put(20,30){\line(0,-1){30}}\put(30,30){\line(0,-1){30}}\put(40,30){\line(0,-1){20}}\put(50,30){\line(0,-1){10}}
\put(03,22){{\small $1$}}\put(13,22){{\small $2'$}}\put(23,22){{\small $2$}}\put(33,22){{\small $2$}}\put(43,22){{\small $1^{\circ}$}}
                                       \put(13,12){{\small $2'$}}\put(23,12){{\small $3$}}\put(33,12){{\small $1^{\circ}$}}
                                                                              \put(23,02){{\small $4'$}}
\end{picture}
\ \ 
\begin{picture}(60,40)
\put(00,30){\line(1,0){50}}\put(00,20){\line(1,0){50}}\put(10,10){\line(1,0){30}}\put(20,00){\line(1,0){10}}\put(00,30){\line(0,-1){10}}\put(10,30){\line(0,-1){20}}\put(20,30){\line(0,-1){30}}\put(30,30){\line(0,-1){30}}\put(40,30){\line(0,-1){20}}\put(50,30){\line(0,-1){10}}
\put(03,22){{\small $1$}}\put(13,22){{\small $2'$}}\put(23,22){{\small $2$}}\put(33,22){{\small $1^{\circ}$}}\put(43,22){{\small $2^{\circ}$}}
                                       \put(13,12){{\small $2'$}}\put(23,12){{\small $3$}}\put(33,12){{\small $1^{\circ}$}}
                                                                              \put(23,02){{\small $4'$}}
\end{picture}
\end{center}
The following are non-examples due to rules (2), (5), (4), respectively:
\begin{center}
\begin{picture}(60,40)
\put(00,30){\line(1,0){50}}\put(00,20){\line(1,0){50}}\put(10,10){\line(1,0){30}}\put(20,00){\line(1,0){10}}\put(00,30){\line(0,-1){10}}\put(10,30){\line(0,-1){20}}\put(20,30){\line(0,-1){30}}\put(30,30){\line(0,-1){30}}\put(40,30){\line(0,-1){20}}\put(50,30){\line(0,-1){10}}
\put(03,22){{\small $1$}}\put(13,22){{\small $2'$}}\put(23,22){{\small $2$}}\put(33,22){{\small $2$}}\put(43,22){{\small $3'$}}
                                       \put(13,12){{\small $2'$}}\put(23,12){{\small $2$}}\put(33,12){{\small $4$}}
                                                                              \put(23,02){{\small $4'$}}
\end{picture}
\ \ 
\begin{picture}(60,40)
\put(00,30){\line(1,0){50}}\put(00,20){\line(1,0){50}}\put(10,10){\line(1,0){30}}\put(20,00){\line(1,0){10}}\put(00,30){\line(0,-1){10}}\put(10,30){\line(0,-1){20}}\put(20,30){\line(0,-1){30}}\put(30,30){\line(0,-1){30}}\put(40,30){\line(0,-1){20}}\put(50,30){\line(0,-1){10}}
\put(03,22){{\small $1$}}\put(13,22){{\small $2'$}}\put(23,22){{\small $2$}}\put(33,22){{\small $2$}}\put(43,22){{\small $1^{\circ}$}}
                                       \put(13,12){{\small $2'$}}\put(23,12){{\small $3$}}\put(33,12){{\small $1^{\circ}$}}
                                                                              \put(23,02){{\small $1^{\circ}$}}
\end{picture}
\ \ 
\begin{picture}(60,40)
\put(00,30){\line(1,0){50}}\put(00,20){\line(1,0){50}}\put(10,10){\line(1,0){30}}\put(20,00){\line(1,0){10}}\put(00,30){\line(0,-1){10}}\put(10,30){\line(0,-1){20}}\put(20,30){\line(0,-1){30}}\put(30,30){\line(0,-1){30}}\put(40,30){\line(0,-1){20}}\put(50,30){\line(0,-1){10}}
\put(03,22){{\small $1$}}\put(13,22){{\small $2'$}}\put(23,22){{\small $2$}}\put(33,22){{\small $1^{\circ}$}}\put(43,22){{\small $1^{\circ}$}}
                                       \put(13,12){{\small $2'$}}\put(23,12){{\small $3$}}\put(33,12){{\small $1^{\circ}$}}
                                                                              \put(23,02){{\small $4'$}}
\end{picture}
\end{center}
\end{exm}

We call the alphabet assigned to a box of $\lambda$ by $T$ an {\it entry} of $T$, and denote it by $e\in T$. Abusing the notation slightly, we often write those entries by their assigned alphabets and denote the numeric value of an entry $e\in T$ by $|e|$, {\it i.e.}, $k,k',k^{\circ}\in T$ and $|k|=|k'|=|k^{\circ}|=k$. Let $c(e)$ and $r(e)$ be the column and row indices of an entry $e$ respectively.  Let $\MST(\lambda,f)$ be the set of all marked shifted tableaux of $(\lambda,f)$. If $f=(0,\dots, 0)$, we denote $\MST(\lambda)$ instead of $\MST(\lambda,f)$.

\begin{defn}\label{df2-1}
Consider the infinite sequence of variables $x=(x_i)_{i\in \NN}, z=(z_i)_{i\in \NN}, b=(b_i)_{i\in \NN}$ as before.  Let $(\lambda,f)$ be a flagged strict partition. To each $T \in \MST(\lambda,f)$, we assign the {\it weight}
\[
(xz|b)^T = \prod_{k \in T}\left(x_k+b_{c(k)-r(k)}\right)\cdot \prod_{k' \in T}\left(x_k-b_{c(k')-r(k')}\right) \cdot \prod_{k^{\circ} \in T} \left(z_k+b_{k+r(k^{\circ})-c(k^{\circ})}\right)
\]
where we set $b_{-i}:=-b_{i+1}$ for all $i\geq0$. We define the {\it flagged factorial $Q$-function} $Q_{\lambda,f}(x;z|b)$ by
\[
Q_{\lambda,f}(x;z|b) = \sum_{T\in \MST(\lambda,f)} (xz|b)^T.
\]
\end{defn}
\begin{rem}
When $f=(0,\dots, 0)$, the $z$-variables are not involved and $Q_{\lambda,f}(x;z|b)$ coincides with Ivanov's factorial $Q$-function $Q_{\lambda}(x|b)$ \cite{Ivanov04}:
\[
Q_{\lambda}(x|b) = \sum_{T\in \MST(\lambda)} (x|b)^T, \ \ \ (x|b)^T = \prod_{k \in T}\left(x_k+b_{c(k)-r(k)}\right)\cdot \prod_{k' \in T}\left(x_k-b_{c(k')-r(k')}\right).
\]
Furthermore, in view of Remark \ref{remMST}, $Q_{\lambda,f}(x;z|b)$ can be expanded in terms of $Q_{\mu}(x|b)$ for strict partitions $\mu\subset \lambda$. This expansion will be discussed in the next subsection.
\end{rem}
\begin{exm}\label{remMST2}
Let $\lambda=(3,1)$ and $f=(1,0)$. In this case, $\MST(\lambda,f)$ can be divided into two families of tableaux 
\setlength{\unitlength}{0.5mm}
\begin{center}
\begin{picture}(30,20)
\put(00,20){\line(1,0){30}}\put(00,10){\line(1,0){30}}\put(10,00){\line(1,0){10}}\put(00,20){\line(0,-1){10}}\put(10,20){\line(0,-1){20}}\put(20,20){\line(0,-1){20}}\put(30,20){\line(0,-1){10}}
\put(03,12){{\small $*$}}\put(13,12){{\small $*$}}\put(23,12){{\small $*$}}
                                      \put(13,02){{\small $*$}}
\end{picture}
\ \ \ \ \ \ 
\begin{picture}(30,20)
\put(00,20){\line(1,0){30}}\put(00,10){\line(1,0){30}}\put(10,00){\line(1,0){10}}\put(00,20){\line(0,-1){10}}\put(10,20){\line(0,-1){20}}\put(20,20){\line(0,-1){20}}\put(30,20){\line(0,-1){10}}
\put(03,12){{\small $*$}}\put(13,12){{\small $*$}}\put(22,12){{\small $1^{\circ}$}}
                                      \put(13,02){{\small $*$}}
\end{picture}
\end{center}
where the part with $*$ consists of unmarked and primed numbers. Thus we have 
\begin{eqnarray*}
Q_{\lambda,f}(x;z|b) &=& Q_{31}(x|b) + Q_{21}(x|b) (z_1 - b_2) 
\end{eqnarray*}
Similarly, if $\lambda=(5,3,1)$ and $f=(2,1,0)$, we have
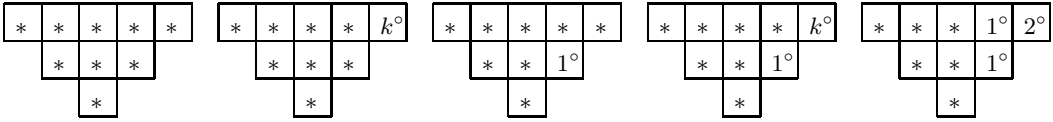
\setlength{\unitlength}{0.5mm}
\begin{center}
\begin{picture}(50,40)
\put(00,30){\line(1,0){50}}\put(00,20){\line(1,0){50}}\put(10,10){\line(1,0){30}}\put(20,00){\line(1,0){10}}\put(00,30){\line(0,-1){10}}\put(10,30){\line(0,-1){20}}\put(20,30){\line(0,-1){30}}\put(30,30){\line(0,-1){30}}\put(40,30){\line(0,-1){20}}\put(50,30){\line(0,-1){10}}
\put(03,22){{\small $*$}}\put(13,22){{\small $*$}}\put(23,22){{\small $*$}}\put(33,22){{\small $*$}}\put(43,22){{\small $*$}}
                                       \put(13,12){{\small $*$}}\put(23,12){{\small $*$}}\put(33,12){{\small $*$}}
                                                                              \put(23,02){{\small $*$}}
\end{picture}
\ \ 
\begin{picture}(50,40)
\put(00,30){\line(1,0){50}}\put(00,20){\line(1,0){50}}\put(10,10){\line(1,0){30}}\put(20,00){\line(1,0){10}}\put(00,30){\line(0,-1){10}}\put(10,30){\line(0,-1){20}}\put(20,30){\line(0,-1){30}}\put(30,30){\line(0,-1){30}}\put(40,30){\line(0,-1){20}}\put(50,30){\line(0,-1){10}}
\put(03,22){{\small $*$}}\put(13,22){{\small $*$}}\put(23,22){{\small $*$}}\put(33,22){{\small $*$}}\put(43,22){{\small $k^{\circ}$}}
                                       \put(13,12){{\small $*$}}\put(23,12){{\small $*$}}\put(33,12){{\small $*$}}
                                                                              \put(23,02){{\small $*$}}
\end{picture}
\ \ 
\begin{picture}(50,40)
\put(00,30){\line(1,0){50}}\put(00,20){\line(1,0){50}}\put(10,10){\line(1,0){30}}\put(20,00){\line(1,0){10}}\put(00,30){\line(0,-1){10}}\put(10,30){\line(0,-1){20}}\put(20,30){\line(0,-1){30}}\put(30,30){\line(0,-1){30}}\put(40,30){\line(0,-1){20}}\put(50,30){\line(0,-1){10}}
\put(03,22){{\small $*$}}\put(13,22){{\small $*$}}\put(23,22){{\small $*$}}\put(33,22){{\small $*$}}\put(43,22){{\small $*$}}
                                       \put(13,12){{\small $*$}}\put(23,12){{\small $*$}}\put(33,12){{\small $1^{\circ}$}}
                                                                              \put(23,02){{\small $*$}}
\end{picture}
\ \ 
\begin{picture}(50,40)
\put(00,30){\line(1,0){50}}\put(00,20){\line(1,0){50}}\put(10,10){\line(1,0){30}}\put(20,00){\line(1,0){10}}\put(00,30){\line(0,-1){10}}\put(10,30){\line(0,-1){20}}\put(20,30){\line(0,-1){30}}\put(30,30){\line(0,-1){30}}\put(40,30){\line(0,-1){20}}\put(50,30){\line(0,-1){10}}
\put(03,22){{\small $*$}}\put(13,22){{\small $*$}}\put(23,22){{\small $*$}}\put(33,22){{\small $*$}}\put(43,22){{\small $k^{\circ}$}}
                                       \put(13,12){{\small $*$}}\put(23,12){{\small $*$}}\put(33,12){{\small $1^{\circ}$}}
                                                                              \put(23,02){{\small $*$}}
\end{picture}
\ \ 
\begin{picture}(50,40)
\put(00,30){\line(1,0){50}}\put(00,20){\line(1,0){50}}\put(10,10){\line(1,0){30}}\put(20,00){\line(1,0){10}}\put(00,30){\line(0,-1){10}}\put(10,30){\line(0,-1){20}}\put(20,30){\line(0,-1){30}}\put(30,30){\line(0,-1){30}}\put(40,30){\line(0,-1){20}}\put(50,30){\line(0,-1){10}}
\put(03,22){{\small $*$}}\put(13,22){{\small $*$}}\put(23,22){{\small $*$}}\put(33,22){{\small $1^{\circ}$}}\put(43,22){{\small $2^{\circ}$}}
                                       \put(13,12){{\small $*$}}\put(23,12){{\small $*$}}\put(33,12){{\small $1^{\circ}$}}
                                                                              \put(23,02){{\small $*$}}
\end{picture}
\end{center}
where $k=1$ or $2$ so that
\begin{eqnarray*}
Q_{\lambda,f}(x|b) 
&=& Q_{531}(x|b) + Q_{431}(x|b) (z_1-b_4 + z_2-b_3) + Q_{521}(x|b)(z_1-b_2) \\
&&\ \ \ \ \ \ \   + Q_{421}(x|b)(z_1-b_4 + z_2-b_3)(z_1-b_2)+ Q_{321}(x|b)(z_1-b_3)(z_2-b_3)(z_1-b_2).
\end{eqnarray*}
\end{exm}
\subsection{Decomposition into $Q$-functions and skew Schur polynomials}
We can expand $Q_{\lambda}(x;z|b)$ in terms of Ivanov's factorial $Q$ functions in $x$ and $b$ where the coefficients are a variant of row-strict flagged skew Schur polynomials considered by Wachs \cite[p.288]{Wachs} in $z$ and $b$. 

A {\it partition} $\lambda$ is a weakly decreasing finite sequence of positive integers and we identify it with its Young diagram. The length of $\lambda$ is $r$ if it consists of $r$ positive integers. We denote the set of all partitions by $\calP$. Let $\lambda=(\lambda_1,\dots, \lambda_r)$ and $\mu=(\mu_1,\dots, \mu_r)$ be partitions of length at most $r$ such that $\mu\subset \lambda$, and $\lambda/\mu$ the corresponding skew diagram. A flagging $f=(f_1,\dots, f_r)$ of $\lambda/\mu$ is a sequence of non-positive integers. We call the pair $(\lambda/\mu,f)$ a {\it flagged skew diagram}. A {\it row-strict (flagged) tableau} $T$ of $(\lambda/\mu, f)$ is a filling of the skew diagram $\lambda/\mu$ which assigns  a positive integer to each box of $\lambda/\mu$ subject to the rules:
\begin{itemize}
\item numbers increase strictly from left to right along rows,
\item numbers increase weakly from top to bottom along columns, and
\item for each $i=1,\dots,r$, the numbers used in the $i$-th row are at most $f_i$.
\end{itemize}
Let $\SST^*(\lambda/\mu,f)$ be the set of all row-strict tableaux of the flagged skew diagram $(\lambda/\mu, f)$. 

\begin{defn}\label{df: row Schur} 
Let $\bfb=(b_i)_{i\in \ZZ}$. We define the row-strict flagged skew factorial Schur polynomial of a flagged shape $(\lambda/\mu,f)$ by
\[
\widetilde{s}_{\lambda/\mu,f}(z|\bfb) = \sum_{T\in \SST^*(\lambda/\mu,f)} (z|\bfb)^T
\]
where we assign the weight for each $T$ by 
\[
(z|\bfb)^T = \prod_{e\in T} \big(z_{|e|} + b_{|e|+r(e)-c(e)}\big).
\]
Note that here we {\it do not} assume $b_{-i}=-b_{i+1}$ for $i\geq 0$.  In the case when $\mu=\varnothing$, then we denote the corresponding polynomial by $\widetilde{s}_{\lambda,f}(z|\bfb)$. 
\end{defn}

\begin{rem}
In \S \ref{app1}, we give a Jacobi--Trudi type determinant formula for the row-strict flagged skew factorial Schur polynomials (Theorem \ref{thm:app1}). The proof uses the lattice path method.
\end{rem}
\begin{prop}\label{prop1}
Let $\lambda$ be a strict partition of length $r$. For a strict partition $\mu\subset \lambda$, let $\bar\mu=(\bar\mu_1,\dots,\bar\mu_r)$ be the sequence defined by $\bar\mu_i=\mu_i+i-1$ for $i=1,\dots, r$.  Assume that, if $r\geq 2$, then  $\lambda_{r-1} >f_{r-1}$ or $\lambda_r> f_r$.  Then we have
\begin{equation}\label{eqprop1}
Q_{\lambda,f}(x;z|b) = \sum_{\mu \in \SP \atop{\mu\subset\lambda \atop{\bar\mu\in \calP}}} Q_{\mu}(x|b)\cdot \widetilde{s}_{\bar\lambda/\bar\mu, f}(z|\bfb)^{\star},
\end{equation}
where $\star$ is the substitution $b_{-i}\mapsto -b_{i+1}$ for all $i\geq 0$. 
\end{prop}
\begin{proof}
The circled numbers form a row-strict flagged skew tableau of a skew {\it shifted} diagram $\lambda/\mu$ since the alphabets must be weakly increasing in each row and column. The assumption assures that this skew shifted diagram is indeed a skew (unshifted) diagram $\bar\lambda/\bar\mu$, {\it i.e.}, $\bar\mu$ is a partition contained in the partition $\bar \lambda$. Thus we see that there is an obvious bijection
\[
\MST(\lambda,f) \cong \bigsqcup_{\mu\in \SP \atop{\mu \subset \lambda \atop{\bar\mu \in\calP}} }\MST(\mu) \times \SST^*(\bar\lambda/\bar\mu,f), \ \ \ \ \ \ T \mapsto (T',T^{\circ})
\]
where $T'$ is the part of $T$ with unmarked and primed numbers and $T^{\circ}$ is the part of $T$ with circled numbers. This bijection apparently preserves the weights after the substitution $\star$, and hence we obtain the desired formula.
\end{proof}
\begin{rem}
Proposition \ref{prop1} implies that $Q_{\lambda,f}(x;z|b)$ is an element of $\Gamma[z,b]$ defined in \S\ref{secprelim}. Indeed, it follows from the facts that the summation in (\ref{eqprop1}) is finite, and that both $Q_{\mu}(x|b)$ and $\widetilde{s}_{\bar\lambda/\bar\mu, f}(z|\bfb)^{\star}$ are elements of $\Gamma[z,b]$.
\end{rem}
\subsection{One row case}
In this section, we describe $Q_{\lambda,f}(x;z|b)$ in the case $\lambda$ has only one row. 
\begin{lem}\label{lem1}
Let $r,t$ and $f$ be nonnegative integers such that $r-t\geq 0$. We have
\[
\widetilde{s}_{(r)/(t), (f)}(z|\bfb) = e_{r-t}^{[f|r-t-f-1]}(z|\tau^{-t}b).
\]
In particular, if $r-t>f$, both sides of the equation are zero.
\end{lem}
\begin{proof}
If $r-t>f$, then the left hand side is zero, since the tableaux are row-strict. The right hand side is also zero, since it is the $(r-t)$-th elementary symmetric polynomial in $r-t-1$ variables. Suppose $r-t\leq f$. If $t=0$, then we have
\begin{eqnarray*}
\widetilde{s}_{(r), (f)}(z|\bfb)
&=&\sum_{1\leq i_1<\cdots <i_r\leq f} (z_{i_1}+b_{i_1})(z_{i_2}+b_{i_2-1})\cdots (z_{i_r}+b_{i_r+1-r})\\
&=&\sum_{1\leq j_1\leq\cdots \leq j_r\leq f+1-r} (b_{j_1}+z_{j_1})(b_{j_2}+z_{j_2+1})\cdots (b_{j_r}+z_{j_r+r-1}).
\end{eqnarray*}
Since this is the usual one-row factorial Schur polynomial, we have 
\[
\widetilde{s}_{(r), (f)}(z|\bfb) = h_r^{[f+1-r|-f]}(b|z) = e_r^{[f|r-f-1]}(z|b).
\]
In the general case $t\geq 0$, let $m:=r-t$, then we have
\begin{eqnarray*}
\widetilde{s}_{(r)/(t), (f)}(z|\bfb)
&=&\sum_{1\leq i_1<\cdots <i_{m}\leq f} (z_{i_1}+b_{i_1-t})(z_{i_2}+b_{i_2-1-t})\cdots (z_{i_m}+b_{i_m+1-m-t})\\
&=&\widetilde{s}_{(m), (f)}(z|\tau^{-t}\bfb)= e_m^{[f|m-f-1]}(z|\tau^{-t}b).
\end{eqnarray*}
This completes the proof of the formula.
\end{proof}
\begin{rem}
Suppose that $0\leq r-t \leq f$. The $b$-variables appearing in $\widetilde{s}_{(r)/(t), (f)}(z|\bfb)$ are
\[
b_{1-t}, b_{2-t},\dots,b_{f-r} ,b_{f-r+1}.
\]
If $r>f$, then $t>0$ and the indices of those $b_i$'s are all nonpositive. 
\end{rem}
Proposition \ref{prop1} and Lemma \ref{lem1} imply the following.
\begin{prop}\label{prop2-1}
For nonnegative integers $r$ and $f$, we have  
\[
Q_{(r), (f)}(x;z|b) = \sum_{k=0}^{f} q_{r-k}^{[r-k-1]}(x|b) \cdot   e_k^{[f|k-f-1]}(z|\tau^{k-r}b)^{\star},
\]
where $\star$ is the substitution $b_{-i}=-b_{i+1}$ for all $i\geq 0$. 
\end{prop}
\begin{proof}
Proposition \ref{prop1} implies that
\[
Q_{(r), (f)}(x;z|b) = \sum_{k=0}^r Q_{(r-k)}(x|b) \cdot \widetilde{s}_{(r)/(r-k), (f)}(z|\bfb)^{\star}.
\]
It is known that $Q_{(m)}(x|b)=q_m^{[m-1]}(x|b)$ (see \cite[\S11]{IkedaMihalceaNaruse}) and thus together with Lemma \ref{lem1} we have
\[
Q_{(r), (f)}(x;z|b) = \sum_{k=0}^{r} q_{r-k}^{[r-k-1]}(x|b) \cdot   e_k^{[f|k-f-1]}(z|\tau^{k-r}b)^{\star}.
\]
The upper bound for $k$ in the summation can be $f$ instead of $r$: if $r<f$, the claim holds since $q_{r-k}^{[r-k-1]}(x|b)=0$ for $r<k\leq f$; if $f<r$, the claim holds since $e_k^{[f|k-f-1]}(z|\tau^{k-r}b) = 0$ for $f<k\leq r$. Thus we have proved the desired formula.
\end{proof}
\subsection{Other formulas}
In the rest of the section, we prove Proposition \ref{prop2} which will be used in the proof of Theorem \ref{mainthm}. Let $\star$ denote the substitution $b_{-i}\mapsto -b_{i+1}$ for all $i\geq 0$ as before. We start with the following lemma.
\begin{lem}\label{lem2}
Let $s, t, m\in \ZZ$ and $n \in \ZZ_{\geq 0}$. For each $s\in \ZZ$, we have
\[
\sum_{\ell\leq s} q_{\ell}^{[m]}\cdot e_{t-\ell}^{[-n-1]}(\tau^{-m}b)^{\star}
=q_{s}^{[m-1]}\cdot e_{t-s}^{[-n-1]}(\tau^{-m}b)^{\star}
+\sum_{\ell \leq s-1} q_{\ell}^{[m-1]}\cdot e_{t-\ell}^{[-n]}(\tau^{1-m}b)^{\star}.
\]
\end{lem}
\begin{proof}
By definition, we have $q_u^{[m]} = q_u^{[m-1]}\cdot (1+ b_m^{\star}u)$ so that 
\begin{equation}\label{eeqq1}
q_{\ell}^{[m]} = q_{\ell}^{[m-1]} + q_{\ell-1}^{[m]} \cdot b_m^{\star} \ \ \ (\ell\in \ZZ).
\end{equation}
Similarly, we have $e_u^{[-n]}(\tau^{1-m} b)^{\star}=e_u^{[-n-1]}(\tau^{-m}b)^{\star} \cdot (1+b_{m}^{\star} u)$ 
so that
\begin{equation}\label{eeqq2}
e_{\ell}^{[-n]}(\tau^{1-m} b)^{\star} = e_{\ell}^{[-n-1]}(\tau^{-m}b)^{\star} + e_{\ell-1}^{[-n-1]}(\tau^{-m}b)^{\star}\cdot b_{m}^{\star}\ \ \ \ \ (\ell\in \ZZ).
\end{equation}
Using these identities, we can compute:
\begin{eqnarray*}
&&\sum_{\ell\leq s} q_{\ell}^{[m]}\cdot e_{t-\ell}^{[-n-1]}(\tau^{-m}b)^{\star}\\
&\stackrel{(\ref{eeqq1})}{=}&\sum_{\ell\leq s} q_{\ell}^{[m-1]}\cdot e_{t-\ell}^{[-n-1]}(\tau^{-m}b)^{\star}+\sum_{\ell\leq s} q_{\ell-1}^{[m-1]} \cdot b_m^{\star}\cdot e_{t-\ell}^{[-n-1]}(\tau^{-m}b)^{\star}\\
&=&q_{s}^{[m-1]}\cdot e_{t-s}^{[-n-1]}(\tau^{-m}b)^{\star}+\sum_{\ell\leq s-1} q_{\ell}^{[m-1]}\cdot e_{t-\ell}^{[-n-1]}(\tau^{-m}b)^{\star}+\sum_{\ell\leq s-1} q_{\ell}^{[m-1]} \cdot b_m^{\star}\cdot e_{t-\ell-1}^{[-n-1]}(\tau^{-m}b)^{\star}\\
&\stackrel{(\ref{eeqq2})}{=}&q_{s}^{[m-1]}\cdot e_{t-s}^{[-n-1]}(\tau^{-m}b)^{\star}
+\sum_{\ell\leq s-1} q_{\ell}^{[m-1]}\cdot e_{t-\ell}^{[-n]}(\tau^{1-m}b)^{\star}.
\end{eqnarray*}
Thus we obtain the desired formula.
\end{proof}
\begin{prop}\label{prop2}
For integers $r,f\geq 0$ and an integer $a$, we have
\[
q_{r+a}^{[f|r-f-1]}(x;z|b)=\sum_{k=0}^f q_{r-k+a}^{[r-k-1]}(x|b)\cdot e_{k}^{[f|k-1-f]}(z|\tau^{k-r}b)^{\star}.
\]
In particular, we have $q_r^{[f|r-f-1]}(x;z|b)= Q_{(r), (f)}(x;z|b)$ in the view of Proposition \ref{prop2-1}.

\end{prop}
\begin{proof}
First we observe that $e_u^{[r-1-f]}(b)=e_u^{[r]}(b)\cdot e_u^{[-1-f]}(\tau^{-r}b)^{\star}$. Indeed, if $r>f$, then
\begin{eqnarray*}
e_u^{[r]}(b)\cdot e_u^{[-1-f]}(\tau^{-r}b)^{\star} 
&=&e_u^{[r]}(b_1,\dots, b_r)\cdot e_u^{[-1-f]}(b_{1-r},b_{2-r},\dots, \dots, b_{f+1-r})^{\star}\\
&=&e_u^{[r]}(b_1,\dots, b_r)\cdot e_u^{[-1-f]}(-b_r,-b_{r-1},\dots, -b_{r-f})\\
&=&e_u^{[r]}(b_1,\dots, b_r)\cdot e_u^{[f+1]}(b_{r-f}, \cdots, b_{r-1}, b_r)^{-1}\\
&=&e_u^{[r-1-f]}(b_1,\dots, b_{r-f-1}) = e_u^{[r-1-f]}(b).
\end{eqnarray*}
If $r\leq f$, then
\begin{eqnarray*}
e_u^{[r]}(b)\cdot e_u^{[-1-f]}(\tau^{-r}b)^{\star} 
&=&e_u^{[r]}(b_1,\dots, b_r)\cdot e_u^{[-1-f]}(\underbrace{b_{1-r},b_{2-r},\dots, b_{-1},b_0}_{r},b_1\dots, b_{f+1-r})^{\star}\\
&=&e_u^{[r]}(b_1,\dots, b_r)\cdot e_u^{[-1-f]}(\underbrace{-b_r,-b_{r-1},\dots, -b_2,-b_1}_{r},b_1\dots, b_{f+1-r})\\
&=&e_u^{[r-1-f]}(b_1,\dots, b_{f+1-r})=e_u^{[r-1-f]}(b).
\end{eqnarray*}
Thus $q_u^{[f|r-f-1]}=q_u^{[r]}\cdot e_u^{[f|-1-f]}(z|\tau^{-r}b)^{\star}$. In particular, we have
\begin{equation}\label{eq4443}
q_{r+a}^{[f|r-f-1]}=\sum_{\ell\leq r+a} q_\ell^{[r]}\cdot e_{r+a-\ell}^{[f|-1-f]}(z|\tau^{-r}b)^{\star}.
\end{equation}
On the other hand, by setting $s=r+a-k$, $m=r-k$, $t=r+a$, $n=f-k$ for $k=0,\dots,f$ in the identity of Lemma \ref{lem2}, we obtain
\begin{eqnarray*}
&&\sum_{\ell\leq r+a-k} q_{\ell}^{[r-k]}\cdot e_{r+a-\ell}^{[f|k-1-f]}(z|\tau^{k-r}b)^{\star}\\
&=&q_{r+a-k}^{[r-k-1]}\cdot e_{k}^{[f|k-1-f]}(z|\tau^{k-r}b)^{\star}+\sum_{\ell \leq r+a-k-1} q_{\ell}^{[r-k-1]}\cdot e_{r+a-\ell}^{[f|k-f]}(z|\tau^{k+1-r}b)^{\star}.
\end{eqnarray*}
We apply this to the right hand side of (\ref{eq4443}) consecutively from $k=0$ to $k=f$, and obtain
\begin{eqnarray*}
q_{r+a}^{[f|r-f-1]}
&=&\sum_{\ell\leq r+a} q_\ell^{[r]}\cdot e_{r+a-\ell}^{[f|-1-f]}(z|\tau^{-r}b)^{\star}\\
&=&q_{r+a}^{[r-1]}\cdot e_{0}^{[f|-1-f]}(z|\tau^{-r}b)^{\star}+\sum_{\ell \leq r+a-1} q_{\ell}^{[r-1]}\cdot e_{r+a-\ell}^{[f|-f]}(z|\tau^{1-r}b)^{\star}\\
&=&\sum_{k=0}^1 q_{r+a-k}^{[r-1-k]}\cdot e_{k}^{[f|k-1-f]}(z|\tau^{k-r}b)^{\star}+\sum_{\ell \leq r+a-2} q_{\ell}^{[r-2]}\cdot e_{r+a-\ell}^{[f|1-f]}(z|\tau^{2-r}b)^{\star}\\
&=&\cdots\\
&=&\sum_{k=0}^f q_{r+a-k}^{[r-1-k]}\cdot e_{k}^{[f|k-1-f]}(z|\tau^{k-r}b)^{\star}+\sum_{\ell \leq r+a-f-1} q_{\ell}^{[r-f-1]}\cdot e_{r+a-\ell}^{[f|0]}(z|\tau^{f+1-r}b)^{\star}.
\end{eqnarray*}
The last summation is zero since $r+a-\ell>f$ and $e_u^{[f|0]}$ is a degree $f$ polynomial in $u$. Thus we obtain the desired equation.
\end{proof}
\begin{rem}\label{remprop2}
The identity of Proposition \ref{prop2} can be also written as
\[
q_{r+a}^{[f|r-f-1]}(x;z|b)=\sum_{k\in \ZZ} q_{r-k+a}^{[r-k-1]}(x|b)\cdot e_{k}^{[f|k-1-f]}(z|\tau^{k-r}b)^{\star}.
\]
since, if $k>f$, then $e_u^{[f|k-1-f]}(z|\tau^{k-r}b)$ is a degree $k-1$ polynomial in $u$ so that $e_{k}^{[f|k-1-f]}(z|\tau^{k-r}b)=0$.
\end{rem}
\section{Schur-Pfaffian formula}\label{secPf}
In this section, we review the basic properties of Schur-Pfaffian and then prove a Pfaffian formula of the flagged factorial $Q$-functions $Q_{\lambda}(x;z|b)$.
\subsection{Schur-Pfaffian and factorial $Q$-functions}
Let $\alpha=(\alpha_1,\dots,\alpha_r) \in \ZZ^r$ be a sequence of integers. Consider the Laurent series in variables $t_1,\dots, t_r$ 
\[
f^{\alpha}(t)=t_1^{\alpha_1}\cdots t_r^{\alpha_r} \prod_{1\leq i<j\leq r}\frac{1-t_i/t_j}{1+t_i/t_j}
\]
where we expand $\frac{1}{1+t_i/t_j}$ as the series $\sum_{m\geq 0} (- t_it_j^{-1})^m$. Consider sequences of indeterminants
\[
c^{(i)}= \left(\, c^{(i)}_m\, \right)_{m \in \ZZ} \ \ \ \ \  \ (i=1,\dots,r).
\]  
The {\it Schur-Pfaffain} $\Pf\left[c^{(1)}_{\alpha_1}\cdots c^{(r)}_{\alpha_r}\right]$ associated to $\alpha$ is defined by replacing each monomial $t_1^{m_1}\cdots t_r^{m_r}$ in $f^{\alpha}(t)$ by $c^{(1)}_{m_1}\cdots c^{(r)}_{m_r}$.

Below in Lemma \ref{lem3} and \ref{lem4}, we list well-known properties without proofs (cf. \cite{IkedaMatsumura}).
\begin{lem}\label{lem3}\ 
\begin{enumerate}[$(1)$]
\item  If $\Pf[c^{(i)}_{\alpha_i}c^{(j)}_{\alpha_j}]+\Pf[c^{(j)}_{\alpha_j}c^{(i)}_{\alpha_i}]=0$ for all $1\leq i,j\leq r$, then for any $w\in S_r$, we have
\[
\Pf\left[c^{(1)}_{\alpha_1}\cdots c^{(r)}_{\alpha_r}\right] = \sign(w) \Pf\left[c^{(1)}_{\alpha_{w(1)}}\cdots c^{(r)}_{\alpha_{w(r)}}\right].
\]
\item If $c_m^{(i)}=k a_m + \ell b_m$ with variables $a=(a_m)_{m\in \ZZ}$ and $b=(b_m)_{m\in \ZZ}$, we have
\[
\Pf\left[c^{(1)}_{\alpha_1}\cdots c^{(i)}_{\alpha_i} \cdots c^{(r)}_{\alpha_r}\right] = k \Pf\left[c^{(1)}_{\alpha_1}\cdots a_{\alpha_i} \cdots c^{(r)}_{\alpha_r}\right] + \ell \Pf\left[c^{(1)}_{\alpha_1}\cdots b_{\alpha_i} \cdots c^{(r)}_{\alpha_r}\right].
\]
\item If $r$ is even, then
\[
\Pf\left[c^{(1)}_{\alpha_1}\cdots c^{(r)}_{\alpha_r}\right] = \Pf\left( \Pf[c^{(i)}_{\alpha_i}c^{(j)}_{\alpha_j}]   \right)_{1 \leq i<j\leq r}
\]
where the right hand side is the Pfaffian of the $r\times r$ skew symmetric matrix $\left(  \Pf[c^{(i)}_{\alpha_i}c^{(j)}_{\alpha_j}]   \right)_{1 \leq i<j\leq r}$ with $(i,j)$-entry 
\[
\Pf[c^{(i)}_{\alpha_i}c^{(j)}_{\alpha_j}] = c^{(i)}_{\alpha_i}c^{(j)}_{\alpha_j} + 2\sum_{k\geq 1} (-1)^k c^{(i)}_{\alpha_i+k}c^{(j)}_{\alpha_j-k}
\]
for $i<j$
\end{enumerate}
\end{lem}
\begin{rem}
Lemma \ref{lem3} (3) follows from the identity 
\[
\prod_{1\leq i<j\leq r}\frac{1-t_i/t_j}{1+t_i/t_j} = \Pf \left( \frac{1-t_i/t_j}{1+t_i/t_j}\right)_{1\leq i<j \leq r}.
\]
for $r$ even, which is due to Schur \cite{Schur}.
\end{rem}
\begin{lem}\label{lem4}
We denote the substitution $c_m^{[i]} = 0$ for all $m<0$ and $i=1,\dots,r$ by $\displaystyle\Pf\left[c^{(1)}_{\alpha_1}\cdots c^{(r)}_{\alpha_r}\right]_{\geq 0}$. We have
\begin{enumerate}[$(1)$]
\item $\displaystyle\Pf\left[c^{(1)}_{\alpha_1}\cdots c^{(r)}_{\alpha_r}\right]_{\geq 0}$ is a polynomial in $c_m^{(i)}$'s ($m\geq 0$).
\item If $\alpha_r=0$, then $\displaystyle\Pf\left[c^{(1)}_{\alpha_1}\cdots c^{(r)}_{\alpha_r}\right]_{\geq 0}=\displaystyle\Pf\left[c^{(1)}_{\alpha_1}\cdots c^{(r)}_{\alpha_{r-1}}\right]_{\geq 0}$.
\end{enumerate}
\item If $\alpha_r<0$, then $\displaystyle\Pf\left[c^{(1)}_{\alpha_1}\cdots c^{(r)}_{\alpha_r}\right]_{\geq 0}=0$.
\end{lem}
By the work of Kazarian \cite{Kazarian} and Ikeda \cite{Ikeda2007}, it is known that the factorial $Q$-functions $Q_{\lambda}(x|b)$ of Ivanov \cite{Ivanov04} can be expressed as a Schur-Pfaffian: for a strict partition $\lambda=(\lambda_1,\dots,\lambda_r)$, 
\[
Q_{\lambda}(x|b) = \Pf\left[q_{\lambda_1}^{[\lambda_1-1]}q_{\lambda_2}^{[\lambda_2-1]}\cdots q_{\lambda_r}^{[\lambda_r-1]}\right]:=\left.\Pf\left[c^{(1)}_{\lambda_1}c^{(2)}_{\lambda_2}\cdots c^{(r)}_{\lambda_r}\right]\right|_{c^{(i)}_{m}=q_{m}^{[\lambda_i-1]}, \forall i,m}.
\]
\begin{lem}\label{lem5}
For $k, \ell\in \ZZ_{\geq 1}$, we have
\[
\Pf\left[q_{k}^{[k-1]}q_{\ell}^{[\ell-1]}\right]+ \Pf\left[q_{\ell}^{[\ell-1]}q_{k}^{[k-1]}\right]=0.
\]
\end{lem}
\begin{proof}
The left hand side equals to $2\sum_{r\in \ZZ}(-1)^rq_{k+r}^{[k-1]} q_{\ell-r}^{[\ell-1]}$, 
which is the coefficient of $u^{k+\ell}$ in $2 q_{-u}^{[k-1]}q_{u}^{[\ell-1]}= 2  e_{-u}^{[k-1]}(b)e_{u}^{[\ell-1]}(b)$,  a polynomial in $u$ of degree $k+\ell-2$. Therefore it is zero.
\end{proof}
Lemma \ref{lem3} (1) and Lemma \ref{lem5} imply the following.
\begin{lem}\label{lem6}
For a sequence of positive integers $(\alpha_1,\dots,\alpha_r)$ and $w\in S_r$, we have
\[
\Pf\left[q_{\alpha_1}^{[\alpha_1-1]}\cdots q_{\alpha_r}^{[\alpha_r-1]}\right] = \sgn(w)\Pf\left[q_{\alpha_{w(1)}}^{[\alpha_1-1]}\cdots q_{\alpha_{w(r)}}^{[\alpha_r-1]}\right].
\]
If particular, if $\alpha_i=\alpha_j$ for some $i\not=j$, we have $\Pf\left[q_{\alpha_1}^{[\alpha_1-1]}\cdots q_{\alpha_r}^{[\alpha_r-1]}\right]=0$.
\end{lem}
\subsection{Schur-Pfaffian formula of flagged factorial $Q$-functions}
\begin{thm}\label{mainthm}
Let $(\lambda,f)$ be a flagged strict partition of length $r$ such that $(\mathrm{a})$ $\lambda_i-f_i \geq \lambda_j-f_j$ for all $1\leq i<j \leq r$ and $(\mathrm{b})$ $\lambda_{r-1}-f_{r-1}>0$.  We have
\[
Q_{\lambda,f}(x,z|b)=\Pf\left[q_{\lambda_1}^{[f_1|\lambda_1-f_1-1]}q_{\lambda_2}^{[f_2|\lambda_2-f_2-1]}\dots q_{\lambda_r}^{[f_r|\lambda_r-f_r-1]}\right].
\]
\end{thm}
\begin{proof}
Let $(\nu_1,\dots, \nu_r) \in \ZZ^r$. By Proposition \ref{prop2} (and Remark \ref{remprop2}),  we have
\[
q_{\lambda_i+\nu_i}^{[f_i|\lambda_i-f_i-1]}
=\sum_{\alpha_i\in \ZZ} q_{\alpha_i+\nu_i}^{[\alpha_i-1]}\cdot e_{\lambda_i-\alpha_i}^{[f_i|\lambda_i-\alpha_i-f_i-1]}(z|\tau^{-\alpha_i}b)^{\star} \ \ \ \ \ \ \ (i=1,\dots,r).
\]
By linearity (Lemma \ref{lem3} (2)),
\begin{eqnarray*}
&&\Pf\left[q_{\lambda_1}^{[f_1|\lambda_1-f_1-1]}q_{\lambda_2}^{[f_2|\lambda_2-f_2-1]}\cdots q_{\lambda_r}^{[f_r|\lambda_r-f_r-1]}\right]\\&=&
\sum_{(\alpha_1,\dots,\alpha_r)\in \ZZ^r} 
\left(\prod_{i=1}^r  e_{\lambda_i-\alpha_i}^{[f_i|\lambda_i-\alpha_i-f_i-1]}(z|\tau^{-\alpha_i}b)^{\star}   \right)
\Pf\left[q_{\alpha_1}^{[\alpha_1-1]}q_{\alpha_2}^{[\alpha_2-1]}\cdots q_{\alpha_r}^{[\alpha_r-1]}\right].
\end{eqnarray*}

Suppose that $\lambda_r-f_r>0$. In this case, by the assumption (a), we have $\lambda_i-f_i>0$ for all $i=1,\dots,r$ so that $e_{\lambda_i-\alpha_i}^{[f_i|\lambda_i-\alpha_i-f_i-1]}(z|\tau^{-\alpha_i}b)^{\star}=0$ for $\alpha_i\leq 0$. Thus we have
\begin{eqnarray*}
&&\Pf\left[q_{\lambda_1}^{[f_1|\lambda_1-f_1-1]}q_{\lambda_2}^{[f_2|\lambda_2-f_2-1]}\cdots q_{\lambda_r}^{[f_r|\lambda_r-f_r-1]}\right]\\&=&
\sum_{(\alpha_1,\dots,\alpha_r)\in (\ZZ_{>0})^r} 
\left(\prod_{i=1}^r  e_{\lambda_i-\alpha_i}^{[f_i|\lambda_i-\alpha_i-f_i-1]}(z|\tau^{-\alpha_i}b)^{\star}   \right)
\Pf\left[q_{\alpha_1}^{[\alpha_1-1]}q_{\alpha_2}^{[\alpha_2-1]}\cdots q_{\alpha_r}^{[\alpha_r-1]}\right].
\end{eqnarray*}
By Lemma \ref{lem6}, we have
\begin{eqnarray*}
&&\Pf\left[q_{\lambda_1}^{[f_1|\lambda_1-f_1-1]}q_{\lambda_2}^{[f_2|\lambda_2-f_2-1]}\cdots q_{\lambda_r}^{[f_r|\lambda_r-f_r-1]}\right]\\
&=&\sum_{\mu\in \SP_r \atop{\mu_r>0}} 
\left(\sum_{w\in S_r} \sgn(w) \prod_{i=1}^r  e_{\lambda_i-\mu_{w(i)}}^{[f_i|\lambda_i-\mu_{w(i)}-f_i-1]}(z|\tau^{-\mu_{w(i)}}b)^{\star}\right)
Q_{\mu}(x|b)\\
&=&\sum_{\mu\in \SP_r \atop{\mu_r>0}} 
\det\left(e_{\lambda_i-\mu_j}^{[f_i|\lambda_i-\mu_j-f_i-1]}(z|\tau^{-\mu_j}b)^{\star}\right)_{1\leq i,j\leq r}
Q_{\mu}(x|b).
\end{eqnarray*}
It is easy to see that the determinant vanishes if there is $k$ such that $\lambda_k-\mu_k<0$. Thus the sum runs over all $\mu\in \SP_r$ such that $\mu_r>0$ and $\mu\subset \lambda$. In particular, $\bar\mu$ is a partition since $\mu_r>0$. Finally, by Theorem \ref{thm:app1}, we have
\begin{eqnarray*}
\det\left(e_{\lambda_i-\mu_j}^{[f_i|\lambda_i-\mu_j-f_i-1]}(z|\tau^{-\mu_j}b)\right)_{1\leq i,j\leq r}
&=&\det\left(e_{\bar\lambda_i-\bar\mu_j+j-i}^{[f_i|\bar\lambda_i-\bar\mu_j+j-i-f_i-1]}(z|\tau^{j-\bar\mu_j-1}b)\right)_{1\leq i,j\leq r}\\
&=&\widetilde{s}_{\bar\lambda/\bar\mu,f}(z|\bfb),
\end{eqnarray*}
where the assumption (a) implies the inequalities that must be satisfied by $(\overline{\lambda},f)$. Thus we obtain
\[
\Pf\left[q_{\lambda_1}^{[f_1|\lambda_1-f_1-1]}q_{\lambda_2}^{[f_2|\lambda_2-f_2-1]}\cdots q_{\lambda_r}^{[f_r|\lambda_r-f_r-1]}\right]\\
=\sum_{\mu\in \SP \atop{\mu\subset \lambda \atop{\bar\mu \in\calP}}} 
\widetilde{s}_{\bar\lambda/\bar\mu,f}(z|\bfb)^{\star}\cdot 
Q_{\mu}(x|b),
\]
and finally the claim follows from Proposition \ref{prop1}. Here note that $\widetilde{s}_{\bar\lambda/\bar\mu,f}(z|\bfb)^{\star}=0$ if $\mu_r=0$ since $\lambda_r-f_r>0$.

Suppose that $\mu_r-f_r\leq 0$. In this case, we have
\begin{eqnarray*}
&&\Pf\left[q_{\lambda_1}^{[f_1|\lambda_1-f_1-1]}q_{\lambda_2}^{[f_2|\lambda_2-f_2-1]}\cdots q_{\lambda_r}^{[f_r|\lambda_r-f_r-1]}\right]\\
&=&
\sum_{(\alpha_1,\dots,\alpha_r)\in (\ZZ_{>0})^r} 
\left(\prod_{i=1}^r  e_{\lambda_i-\alpha_i}^{[f_i|\lambda_i-\alpha_i-f_i-1]}(z|\tau^{-\alpha_i}b)^{\star}   \right)
\Pf\left[q_{\alpha_1}^{[\alpha_1-1]}q_{\alpha_2}^{[\alpha_2-1]}\cdots q_{\alpha_r}^{[\alpha_r-1]}\right]  \\
&&+
\sum_{(\alpha_1,\dots,\alpha_{r-1})\in (\ZZ_{>0})^{r-1}}  e_{\lambda_r}^{[f_r|\lambda_r-f_r-1]}(z|b)^{\star} 
\left(\prod_{i=1}^{r-1}  e_{\lambda_i-\alpha_i}^{[f_i|\lambda_i-\alpha_i-f_i-1]}(z|\tau^{-\alpha_i}b)^{\star}   \right)
\Pf\left[q_{\alpha_1}^{[\alpha_1-1]}q_{\alpha_2}^{[\alpha_2-1]}\cdots q_{\alpha_{r-1}}^{[\alpha_{r-1}-1]}\right] \\
&=&\sum_{\mu\in \SP \atop{\mu\subset \lambda \atop{\mu_r>0}}} \widetilde{s}_{\bar\lambda/\bar\mu,f}(z|\bfb)^{\star}\cdot Q_{\mu}(x|b)+
\sum_{\mu\in \SP \atop{\mu\subset \lambda \atop{\mu_r=0}}} 
e_{\lambda_r}^{[f_r|\lambda_r-f_r-1]}(z|b)^{\star} \det\left(e_{\lambda_i-\mu_j}^{[f_i|\lambda_i-\mu_j-f_i-1]}(z|\tau^{-\mu_j}b)^{\star}\right)_{1\leq i,j\leq r-1}Q_{\mu}(x|b).
\end{eqnarray*}
Since $e_{\lambda_i-\mu_r}^{[f_i|\lambda_i-\mu_r-f_i-1]}(z|\tau^{-\mu_r}b)^{\star}=0$ for all $i=1,\dots,r-1$, we have
\[
e_{\lambda_r}^{[f_r|\lambda_r-f_r-1]}(z|b)^{\star} \det\left(e_{\lambda_i-\mu_j}^{[f_i|\lambda_i-\mu_j-f_i-1]}(z|\tau^{-\mu_j}b)^{\star}\right)_{1\leq i,j\leq r-1}
=\det\left(e_{\lambda_i-\mu_j}^{[f_i|\lambda_i-\mu_j-f_i-1]}(z|\tau^{-\mu_j}b)^{\star}\right)_{1\leq i,j\leq r}.
\]
Thus we obtain
\[
\Pf\left[q_{\lambda_1}^{[f_1|\lambda_1-f_1-1]}q_{\lambda_2}^{[f_2|\lambda_2-f_2-1]}\cdots q_{\lambda_r}^{[f_r|\lambda_r-f_r-1]}\right]\\
=\sum_{\mu\in \SP \atop{\mu\subset \lambda \atop{\bar\mu \in\calP}}} 
\widetilde{s}_{\bar\lambda/\bar\mu,f}(z|\bfb)^{\star}\cdot 
Q_{\mu}(x|b),
\]
and finally the claim follows from Proposition \ref{prop1}. 
\end{proof}
\section{Vexillary double Schubert polynomials of type C}\label{secSchPol}
\subsection{Double Schubert polynomials of type C}
In this section, we briefly recall the double Schubert polynomials of Ikeda--Mihalcea--Naruse. Please see \cite{IkedaMihalceaNaruse} for more detail. 

Let $W_{\infty}$ be the infinite hyperoctahedral group, {\it i.e.}, the Weyl group of type $C_{\infty}$ (or $B_{\infty}$). It is given as the group defined by generators (simple reflections) $\{s_i \ |\ i=0,1,2,\dots\}$ and relations
\[
s_i^2 = e \ (i\geq 0), \ \ s_1s_0s_1s_0=s_0s_1s_0s_1,  \ \ s_is_{i+1}s_i=s_{i+1}s_is_{i+1} (i\geq 1),\ \
s_is_j=s_js_i (|i-j|\geq 2),
\]
where $e$ is the identity element.
We identify $W_{\infty}$ with the group of {\it signed permutations}, {\it i.e.}, permutations $w$ of the set $\{1,2,\dots \} \cup \{-1,-2,\dots\}$ such that $w(i)\not= i$ for only finitely many $i$, and $\overline{w(i)}=w(\bar i)$ where we denote $\bar i = -i$. Each element of $W_{\infty}$, therefore, can be specified by the sequence $(w(1),w(2),\dots)$ which we call the one-line notation of $w$. The simple reflections are identified with the transpositions $s_0=(1,\bar 1)$ and $s_i=(i,i+1)(\bar i, \overline{i+1})$ for $i\geq 1$.

To each $w\in W_{\infty}$, Ikeda--Mihalcea--Naruse \cite{IkedaMihalceaNaruse}  associated a unique function $\frakC_w=\frakC_w(x;z|b)$ in the ring $\Gamma[z,b]$\footnote{Note that the parameters $t=(t_1,t_2,\dots)$ in \cite{IkedaMihalceaNaruse} are replaced by $-b=(-b_1,-b_2,\dots)$ in this paper.}. They are characterized by left and right divided difference operators $\delta_i$ and $\partial_i$ with $i=0,1,2,\dots$. Namely there is a unique family of elements $\frakC_w(x;z|b) \in \Gamma[z,b]$ ($w\in W_{\infty}$), satisfying
\[
\partial_i \frakC_w = \begin{cases}
\frakC_{ws_i} & \mbox{ if } \ell(ws_i)<\ell(w),\\
0 & \mbox{ otherwise},
\end{cases}
\ \ \ \ \ \ 
\delta_i \frakC_w = \begin{cases}
\frakC_{s_iw} & \mbox{ if } \ell(s_iw)<\ell(w),\\
0 & \mbox{ otherwise},
\end{cases}
\] 
for all $i=0,1,2,\dots,$ and such that $\frakC_w$ has no constant term except for $\frakC_e=1$.
\subsection{Vexillary signed permutations}
We follow Anderson--Fulton \cite{AndersonFultonVex}. A {\it triple} is a three $r$-tuples of positive integers, $\tau=(\bfk,\bfp,\bfq)$, with $\bfk=(0<k_1<\cdots<k_r)$,  $\bfp=(p_1\geq \cdots\geq p_r>0)$, and $\bfk=(q_1\geq \cdots\geq q_r>0)$, satisfying the inequality 
\[
(*) \ \ \ \ \ \ k_{i+1}-k_i \leq p_i - p_{i+1} + q_i - q_{i+1}  \ \ \ \ \ \ (1\leq i\leq r-1).
\]
A triple is essential if the inequality ($*$) is strict for all $i$. Each triple reduces to a unique essential triple by successively removing $(k_i,p_i,q_i)$ such that the equality holds in ($*$) and two triples are equivalent if they reduce to the same essential triple.

Anderson--Fulton explained how to construct a signed permutation $w=w(\tau)$ in \cite[\S 2]{AndersonFultonVex} and they define a signed permutation to be {\it vexillary} if it arises from a triple in such a way. Equivalent triples give the same vexillary signed permutation. An essential triple $\tau$ also determines a strict partition $\lambda(\tau)$ of length $r$, by setting $\lambda_{k_i}=p_i+q_i - 1$, and filling in the remaining $\lambda_{k}$ minimally so that $\lambda_1>\cdots >\lambda_r$. Similarly, we introduce a flag $f(\tau)=(f_1,\dots, f_r)$ associated to an essential triple $\tau$ by setting $f_{k_i}:=p_i-1$, and filling in the remaining $f_k$ minimally so that $f_1\geq \cdots \geq f_r$. In this way, we can assign a unique flagged strict partition to each vexillary signed permutation. Note that $m_i:=f_{k_i}$ is nothing but the labeling of $\lambda(\tau)$ given in \cite[\S 4]{AndersonFultonVex}. 

From the work of Anderson-Fulton \cite{AndersonFulton, AndersonFulton2}, it follows that the double Schubert polynomials associated to vexillary signed permutations can be given in the following   Pfaffian formula.
\begin{thm}[Anderson-Fulton \cite{AndersonFulton, AndersonFulton2}]
Let $w$ be a vexillary signed permutation and $(\lambda,f)$ the associated flagged strict partition. Then we have
\[
\frakC_w(x;z|b) = \Pf\left[q_{\lambda_1}^{[f_1|\lambda_1-f_1-1]}q_{\lambda_2}^{[f_2|\lambda_2-f_2-1]}\dots q_{\lambda_r}^{[f_r|\lambda_r-f_r-1]}\right].
\]
\end{thm}
Since, by construction, the flagged strict partition $(\lambda,f)$ associated to a vexillary signed permutation $w$ satisfies the requirement in  Theorem \ref{mainthm}, we obtain the following theorem.
\begin{thm}\label{thmmain2}
Let $w$ be a vexillary signed permutation and $(\lambda,f)$ the associated flagged strict partition. Then we have $\frakC_w(x;z|b) = Q_{\lambda,f}(x;z|b)$.
\end{thm}
\subsection{A new tableau formula of Ivanov's factorial $Q$ functions}
A signed permutation $w$ is {\it Lagrangian} if $w(1)<w(2)<\cdots<w(r)<0<w(r+1)<\cdots$ for some integer $r\geq 1$. A Lagrangian signed permutation is vexillary. Indeed, we can define a triple $\tau$ from which $w$ is constructed by setting $k_i=i,  p_i=1$, and $q_i= \overline{w(i)}$ for $i=1,\dots, r$. The associated flagged strict partition $(\lambda,f)$ is given by $\lambda_i=\overline{w(i)}$ and $f_i=0$ for $i=1,\dots, r$. 

On other hand, if $w$ is vexillary, then $w^{-1}$ is also vexillary. In fact, Anderson--Fulton  showed that for a triple $\tau=(\bfk,\bfp,\bfq)$, we have $w(\tau)^{-1} = w(\tau^*)$ where $\tau^* = (\bfk,\bfq,\bfp)$ (\cite[Lemma 2.3]{AndersonFultonVex}). From this, we can deduce that if $w$ is Lagrangian with the strict partition $\lambda$ of length $r$ (and the flag $(0,\dots,0)$, then $w^{-1}$ is a vexillary signed permutation with the strict partition $\lambda$ and the flag $f=(\lambda_1-1,\dots,\lambda_r-1)$.

It is known (\cite{Kazarian}, \cite{Ikeda2007}) that for a Lagrangian signed permutation $w$ with the associated strict partition $\lambda$, we have $\frakC_w(x;z |b) = Q_{\lambda}(x|b)$. On the other hand, by \cite[Theorem 8.1 (3) ]{IkedaMihalceaNaruse}, we know that for a signed permutation $w$, we have $\frakC_{w}(x;z|b)=\frakC_{w^{-1}}(x;b|z)$, which, by Theorem \ref{thmmain2}, implies that $\frakC_w(x;z|b) = Q_{(\lambda,f)}(x;b|z)$, where $f=(\lambda_1-1,\dots,\lambda_r-1)$. Thus we can conclude that $Q_{\lambda}(x|b) = Q_{(\lambda,f)}(x;b|z)$, which shows that the $z$-variables the right hand side. Now by applying Theorem \ref{mainthm}, we obtain the following theorem.
\begin{thm}\label{thmmain3}
Let $\lambda=(\lambda_1,\dots,\lambda_r)$ be a strict partition of length $r$, and $f=(\lambda_1-1,\dots, \lambda_r-1)$. Then Ivanov's factorial $Q$ function associated to $\lambda$ can be expressed as 
\[
Q_{\lambda}(x|b) = \sum_{T\in \MST(\lambda,f)} (xb)^T, \ \ \ \ \ 
(xb)^T = \prod_{k\in T} x_k\prod_{k'\in T} x_k\prod_{k^{\circ}\in T} b_k.
\]
\end{thm}
\begin{rem}
From Theorem \ref{thmmain3} and Proposition \ref{prop1}, we can also write
\[
Q_{\lambda}(x|b)  = \sum_{\mu \in \SP \atop{\mu\subset\lambda \atop{\bar\mu\in \calP}}} Q_{\mu}(x)\cdot \widetilde{s}_{\bar\lambda/\bar\mu, f}(b),
\] 
for a strict partition $\lambda$ of length $r$ where $f=(\lambda_1-1,\dots,\lambda_r-1)$. In the view of Theorem \ref{thm:app1}, this recovers \cite[Theorem 10.2]{Ivanov04}.
\end{rem}
\section{Appendix: Lattice path method for row-strict Schur polynomials}\label{app1}
In this section, we prove a Jacobi--Trudi type formula (Theorem \ref{thm:app1} below) for the row-strict flagged skew factorial Schur polynomials defined at Definition \ref{df: row Schur}. It is a factorial generalization of Theorem 3.5$^*$ in \cite{Wachs}. We prove it by interpreting the tableaux as lattice paths and applying \cite[Theorem 1.2]{StembridgePf} (cf. \cite{Lindstrom, GesselViennot, GesselViennot2}). 

First we recall the basic notations from \cite{StembridgePf}. Let $D=(V,E)$ be an acyclic oriented graph without multiple edges: $V$ is the set of vertices and $E$ is the set of edges in $D$.  For vertices $u$ and $v$, a path from $u$ to $v$ is a sequence of edges $e_1,\dots, e_m$ such that the source of $e_1$ is $u$, the target of $e_m$ is $v$, and the target of $e_i$ coincides with the source of $e_{i+1}$ for all $i=1,\dots, m-1$. Let $\scP(u,v)$ be the set of all paths from $u$ to $v$. Let $w: E \to R$ be a weight function where $R$ is some commutative ring.  For a path $P$, we also denote $w(P)$ the product of the weights of all edges in $P$. Let 
\[
\GF\left[\scP(u,v)\right] = \displaystyle\sum_{P\in \scP(u,v) } w(P).
\] 
Let $\bfu=(u_1,\dots, u_r)$ and $\bfv=(v_1,\dots, v_r)$ be ordered sets of vertices of $D$. Let $\scP_0(\bfu,\bfv)$ is the set of all non-intersecting $r$-tuples of paths, $\bfP=(P_1,\dots, P_r)$, with $P_i\in \scP(u_i,v_i)$. We denote 
\[
\GF\left[\scP_0(\bfu,\bfv)\right] = \displaystyle\sum_{\bfP \in \scP_0(\bfu,\bfv)} w(\bfP)
\]
where we set $w(\bfP) = w(P_1)w(P_2)\cdots w(P_r)$. Finally, we say that $\bfu$ is $D$-compatible with $\bfv$ if a path $P \in \scP(u_i,v_j)$ intersects with a path $Q\in \scP(u_k,v_l)$ whenever $i<k$ and $j>l$. 
\begin{thm}[Theorem 1.2, \cite{StembridgePf}]\label{appAthm}
Let $\bfu=(u_1,\dots, u_r)$ and $\bfv=(v_1,\dots,v_r)$ be ordered sets of vertices such that $\bfu$ is $D$-compatible with $\bfv$.  Then
\[
\GF\left[\scP_0(\bfu,\bfv)\right] = \det \left(\GF\left[\scP(u_i,v_j)\right]\right)_{1\leq i,j\leq r}.
\]
\end{thm}
In order to apply Theorem \ref{appAthm} to the row-strict flagged Schur polynomials, we introduce an acyclic directed graph $D$ as follows: its vertex set $V$ is $\ZZ\times \ZZ_{\geq 0}$ and there is an edge $(u,v) \in E$ from the source $u$ to the target $v$ if $u-v$ is $(0,1)$ or $(1,1)$. We call an edge $(u,v)$ {\it diagonal} if $u-v=(1,1)$, and {\it vertical} if $u-v=(0,1)$. 

We define a weight function $w: E \to \ZZ[z,\bfb]$ by setting $w(e) = 1$ if $e$ is horizontal and $w(e)=z_t + b_{t-s}$ if $e$ is a diagonal edge with its source at $(s,t)$.

Let $\lambda/\mu$ is a skew (unshifted) diagram of length at most $r$ and $f$ its flag. Consider the ordered sets of vertices $\bfu=(u_1,\dots, u_r)$ and $\bfv=(v_1,\dots, v_r)$ where 
\[
u_i=(\lambda_i-i, f_i), \ \ \ v_i=(\mu_i-i,0).
\]
There is a bijection between $\SST^*(\lambda/\mu,f)$ and $\scP_0(\bfu,\bfv)$ defined as follows. Let $T \in \SST^*(\lambda/\mu,f)$. Let $\bfP=(P_1,\dots, P_r)$ be the corresponding $r$-tuple of paths defined as follows. If $j_m<\cdots<j_1$ are the entries of $i$-th row of $T$  where $m=\lambda_i-\mu_i$, then we define $P_i$ to be  the unique path from $u_i$ to $v_i$ such that the $k$-th diagonal edge has its source at $(\lambda_i-i-k+1, j_k)$ for $k=1,\dots,m$.
For example, let $\lambda=(3,2,1)$, $\mu=(1,1,0)$ and $f=(3,2,1)$. The following is an example of a tableaux $T$ in $\SST^*(\lambda/\mu,f)$ and the corresponding triple of non-intersecting paths.
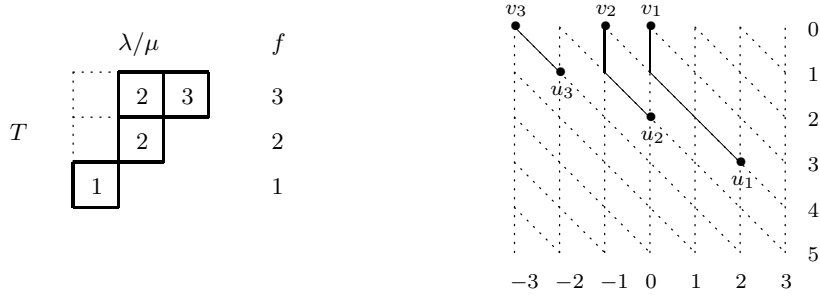
\setlength{\unitlength}{0.6mm}
\begin{center}
\begin{picture}(90,60)

\dottedline{2}(00,40)(30,40)
\dottedline{2}(00,30)(30,30)
\dottedline{2}(00,20)(20,20)
\dottedline{2}(00,10)(10,10)

\dottedline{2}(00,40)(00,10)
\dottedline{2}(10,40)(10,10)
\dottedline{2}(20,40)(20,20)
\dottedline{2}(30,40)(30,30)

\linethickness{0.3mm}
\put(10,40){\line(1,0){20}}
\put(10,30){\line(1,0){20}}
\put(00,20){\line(1,0){20}}
\put(00,10){\line(1,0){10}}
\put(00,20){\line(0,-1){10}}
\put(10,40){\line(0,-1){30}}
\put(20,40){\line(0,-1){20}}
\put(30,40){\line(0,-1){10}}

\put(14,33){{\small $2$}}\put(24,33){{\small $3$}}
\put(14,23){{\small $2$}}
\put(04,13){{\small $1$}}

\put(10,45){{\small $\lambda/\mu$}}
\put(-14,25){{\small $T$}}

\put(44,45){{\small $f$}}
\put(44,33){{\small $3$}}
\put(44,23){{\small $2$}}
\put(44,13){{\small $1$}}
\end{picture}
\ \ \ 
\begin{picture}(50,60)
\put(-2,53){{\footnotesize $v_3$}}
\put(-1,49){{\footnotesize $\bullet$}}
\put(18,53){{\footnotesize $v_2$}}
\put(19,49){{\footnotesize $\bullet$}}
\put(28,53){{\footnotesize $v_1$}}
\put(29,49){{\footnotesize $\bullet$}}

\put(8,35){{\footnotesize $u_3$}}
\put(9,39){{\footnotesize $\bullet$}}
\put(28,25){{\footnotesize $u_2$}}
\put(29,29){{\footnotesize $\bullet$}}
\put(48,15){{\footnotesize $u_1$}}
\put(49,19){{\footnotesize $\bullet$}}

\put(65,48){{\footnotesize $0$}}
\put(65,38){{\footnotesize $1$}}
\put(65,28){{\footnotesize $2$}}
\put(65,18){{\footnotesize $3$}}
\put(65,08){{\footnotesize $4$}}
\put(65,-2){{\footnotesize $5$}}

\put(-1,-8){{\footnotesize $-3$}}
\put(09,-8){{\footnotesize $-2$}}
\put(19,-8){{\footnotesize $-1$}}
\put(29,-8){{\footnotesize $0$}}
\put(39,-8){{\footnotesize $1$}}
\put(49,-8){{\footnotesize $2$}}
\put(59,-8){{\footnotesize $3$}}

\dottedline{2}(00,10)(10,00)
\dottedline{2}(00,20)(20,00)
\dottedline{2}(00,30)(30,00)
\dottedline{2}(00,40)(40,00)
\dottedline{2}(00,50)(50,00)
\dottedline{2}(10,50)(60,00)
\dottedline{2}(20,50)(60,10)
\dottedline{2}(30,50)(60,20)
\dottedline{2}(40,50)(60,30)
\dottedline{2}(50,50)(60,40)

\dottedline{2}(00,00)(00,50)
\dottedline{2}(10,00)(10,50)
\dottedline{2}(20,00)(20,50)
\dottedline{2}(30,00)(30,50)
\dottedline{2}(40,00)(40,50)
\dottedline{2}(50,00)(50,50)
\dottedline{2}(60,00)(60,50)

\linethickness{0.2mm}

\put(20,50){\line(0,-1){10}}
\put(30,50){\line(0,-1){10}}

\put(00,50){\line(1,-1){10}}
\put(20,40){\line(1,-1){10}}
\put(30,40){\line(1,-1){10}}
\put(40,30){\line(1,-1){10}}

\end{picture}
\vspace{5mm}

\end{center}
It is not difficult to see that this defines a bijection from $\SST^*(\lambda/\mu,f)$ to $\scP_0(\bfu,\bfv)$. Moreover, this bijection preserves the weights. Namely, suppose that $T$ corresponds to $\bfP$. Let $j_m < \cdots < j_1$ be the entries of the $i$-th row of $T$. The column index of the entry $j_k$ is $\lambda_i-k+1$ and thus its corresponding weight is $z_{j_k} + b_{j_k + i - (\lambda_i-k+1)}$. On the other hand, $P_i$'s $k$-th diagonal edge has its sources at $(\lambda_i-i-k+1, j_k)$ and thus its weight is also $
z_{j_k} + b_{j_k + i  - (\lambda_i-k+1)}$. For example, the weights of the above examples of a tableau and the corresponding paths are both $(x_2 + b_1)(x_3 + b_1)\cdot (x_2 + b_2) \cdot (x_1 + b_3)$.
Thus we have
\begin{equation}\label{appeq1}
\widetilde{s}_{\lambda/\mu,f}(z|\bfb) = \sum_{T\in \SST^*(\lambda/\mu,f)} (z|\bfb)^T
= \GF\left[\scP_0(\bfu,\bfv)\right].
\end{equation}
The following is an extension of Lemma \ref{lem1} in the view of the lattice path interpretation and will be used in the proof of Theorem \ref{thm:app1} below.
\begin{lem}\label{lemApp1}
Let $u=(s-1,f)$ and $v=(t-1,0)$ where $s,t\in \ZZ$ and $f\in \ZZ_{\leq 0}$, then  we have
\[
\GF\left[\scP(u,v)\right] = e_{s-t}^{[f|s-t-f-1]}(z|\tau^{-t}b).
\]
In particular, this identity is trivially zero unless $0\leq s-t \leq f$.
\end{lem}
\begin{proof}
If $s-t<0$, clearly the identity is zero. If $0 \leq f < s-t$, then $\scP(u,v)=\varnothing$ so that $\GF\left[\scP(u,v)\right]=0$. Furthermore, $e_u^{[f|s-t-f-1]}$ is a polynomial in $u$ of degree $s-t-1$ so that $e_{s-t}^{[f|s-t-f-1]}=0$. Below we suppose that $0\leq s-t \leq f$. 

If $t\geq 0$, the claim follows from Lemma \ref{lem1}.
If $t<0$,  consider $u'=(s-1+n, f)$ and $v'=(t-1+n,0)$ for some $n$ such that $t+n\geq 0$, and then we have, also by Lemma \ref{lem1}, 
\[
\GF\left[\scP(u',v')\right] =  e_{s-t}^{[f|s-t-f-1]}(z|\tau^{-t-n}b).
\]
Since the paths in $\scP(u,v)$ are obtained from the paths in $\scP(u',v')$ by shifting horizontally to the left by $n$ units, we obtain $\GF\left[\scP(u,v)\right]$ from $\GF\left[\scP(u',v')\right]$ by adding $n$ to all indices of $b$ variables. Thus the claim follows.
\end{proof}
\begin{thm}\label{thm:app1}
Let  $(\lambda/\mu,f)$ be a flagged skew partition where $\lambda$ is a partition of length $r$. Assume that $\lambda_i-i-f_i\geq \lambda_j-j-f_j$ for all $i<j$. Then we have
\[
\widetilde{s}_{\lambda/\mu,f}(z|\bfb)=\det\left(e_{\lambda_i-\mu_j+j-i}^{[f_i|\lambda_i-\mu_j+j-i-f_i-1]}(z|\tau^{j-\mu_j-1}b)\right)_{1\leq i,j\leq r}.
\]
\end{thm}
\begin{proof}
By the assumption, it follows that $\bfu$ is $D$-compatible with $\bfv$. Thus we can apply Theorem \ref{appAthm} to the right hand side of (\ref{appeq1}), and obtain
\[
\widetilde{s}_{\lambda/\mu,f}(z|\bfb) = \det \left(\GF\left[\scP(u_i,v_j)\right]\right)_{1\leq i,j\leq r}.
\]
Now the claim follows by applying Lemma \ref{lemApp1} with $u=u_i=(\lambda_i-i,f_i)$ and $v=v_j=(\mu_j-j,0)$ so that $f=f_i$, $s-t=\lambda_i-\mu_j+j-i$, and $t=\mu_j-j+1$.
\end{proof}

\bibliography{references}{}
\bibliographystyle{acm}
\end{document}